\documentclass[a4paper,11pt,twoside]{article}
\usepackage{a4wide}
\usepackage{amssymb, amsmath, amsthm}
\usepackage[utf8]{inputenc}
\usepackage{lmodern}
\usepackage[english]{babel}
\usepackage[notref,notcite,final]{showkeys}
\usepackage{enumerate}
\usepackage{natbib}

\emergencystretch=\hsize
\def\ce{{\mathcal E}}

\def\cl{{\mathcal L}}
\def\cn{{\mathcal N}}
\def\cv{{\mathcal V}}

\def\E{{\mathbb E}}

\def\P{{\mathbb P}}

\def\R{{\mathbb R}}

\def\s{\star}
\def\t{\theta}

\def\ind#1{{\bf 1}_{\left\{#1\right\}}}

\def\abs#1{|#1|}

\def\cov{\mathop{\rm Cov}\nolimits}

\def\Var{\mathop{\rm Var}\nolimits}
\def\Cvar{\mathop{\rm CVaR}\nolimits}
\def\supp{\mathop{\rm supp}\nolimits}

\def\norm#1{\mathop{\left| #1 \right|}\nolimits}

\def\inv#1{\mathop{\frac{1}{ #1}}\nolimits}
\def\expp#1{\mathop {\mathrm{e}^{ #1}}}

\theoremstyle{plain}
\newtheorem{thm}{Theorem}[section]
\newtheorem{proposition}[thm]{Proposition}
\newtheorem{lemma}[thm]{Lemma}

\theoremstyle{definition}

\newtheorem{algo}[thm]{Algorithm}

\let\cite\citet

\title{Rare event simulation for electronic circuit design}
\date{\today}
\author{Xavier Jonsson\footnote{Siemens Electronic Design Automation, 110 Rue Blaise Pascal, 38330 Montbonnot FRANCE,
\newline \indent e-mail: xavier\_jonsson@siemens.com} \and J\'er\^ome Lelong\footnote{Univ. Grenoble Alpes, CNRS, Grenoble INP, LJK, 38000 Grenoble, France, \newline
\indent e-mail: jerome.lelong@univ-grenoble-alpes.fr}}

\begin{document}
\maketitle

\begin{abstract}
  In this work, we propose an algorithm to simulate rare events for electronic circuit design. Our approach heavily relies on a smart use of importance sampling, which enables us to tackle probabilities of the magnitude $10^{-10}$. Not only can we compute very rare default probability, but we can also compute the quantile associated to a given default probability and its expected shortfall. We show the impressive efficiency of method on real circuits.
\end{abstract}

\section{Introduction}

We consider the problem of computing an expectation in the context of rare event simulation using a Monte Carlo approach.
It is well known that, in such a framework, the accuracy of the Monte Carlo estimator is quite bad because too few samples 
lead to a non zero value of the involved characteristics. Assume we want to compute 
$ \E[\varphi(X)] $, where $\varphi : \R^d \longmapsto \R$ and $X$ is a random vector with values in $\R^d$. 
There are basically two well-known tools to improve the accuracy of the crude Monte Carlo estimator: importance sampling and
stratified sampling. As we will see in Section~\ref{sec:is}, importance sampling can be implemented in a quite 
fully automated way while stratification requires an a priori good knowledge of the realization space. 

As an application problem, we will study the  efficient  and  accurate determination   of  
the   yield   of   an   integrated   circuit, through electrical circuit level simulation, 
under stochastic constraints due to the manufacturing process. Several orders of magnitude 
of variance reduction can be achieved, for example millions or billions of MC runs necessary 
for very small probability might be reduced to a few hundreds or thousands runs. Section~\ref{sec:app} 
presents quantitative results and speedup factors obtained with our importance sampling algorithm. 
The performance will be demonstrated on several integrated circuits from SRAM bitcells, to digital 
standard cells, analog IP blocks as well as on some small toy circuits. 

We first give an overview of the semiconductor industry, and of the integrated circuits 
in particular. The notion of yield is explained and how this translates to very low 
probability of failure  estimation for some kind of circuits. We will also introduce shortly the
concept of analog circuit simulation used in our context.

\subsection{EDA market and its segmentation}
\label{sec:intro}
The past decades saw a considerable evolution of the manufacturing processes of integrated circuits (IC). The 
beginning of the twenty-first century experiences the same trend but with the emergence 
of connected objects, applications linked to mobility, onboard systems (especially in the 
automotive domain), the processing of mass data or medical devices. 

For circuit designers, it is necessary to reduce the time of putting new products on the market, while satisfying both the 
demands for increased performance (consumption, reduced weights, sizes and costs). To these constraints related to the end user, 
there are the problems of safety and security (aviation, motor vehicles), connectivity and adaptivity to more and more standards: WLAN 
(Wireless Local Area Network), Bluetooth, Wi-Fi (Wireless Fidelity). The result is an exponential growth in terms of complexity 
of operation and size (in number of transistors) for ICs. 
The time to market of a product is a crucial factor in its success. These systems of great complexity have to be designed 
and validated efficiently. The manufacturing costs being high, multiple attempts cannot be accepted. 

The initial design of complex circuits is obtained in particular by the use of standard cell libraries
and blocks of intellectual property (IP) developed by specialized teams 
or companies. 
Simulation constitutes the initial step after the specification. 
From our perspective, we will consider only a limited 
number of segments of this market: memory circuit simulation and standard cells. The simulation of complete digital blocks (digital SoCs)
or large AMS/RF circuits cannot be addressed with our techniques due to the prohibitive CPU time 
required for each individual simulation. For our experiments, electrical simulations will be 
achieved with the commercial SPICE simulator Eldo\textsuperscript{\tiny\textregistered}. 
SPICE-type simulators (Simulation Program with Integrated Circuit Emphasis)
are the industry-standard way to verify circuit operation at the transistor level before 
manufacturing an integrated circuit.

\subsection{Principles of analog circuit simulation}
\label{sec:analogsim}
The automatic analysis routines of a SPICE simulation program determine the numerical solution of a mathematical
representation of the circuit under validation. 
After each circuit element has been modeled, the system
of equations describing the complete circuit is determined  
and topological constraints are determined by the interconnection of the elements. The equations
of the circuit are generally a nonlinear differential algebraic system of
of the form:
\begin{equation}
  \label{eq:algebrodiff}
  \begin{split}
    \Phi\left(u(t), \dot{u}(t), t\right) = 0 \\
    u(0) = u_0
  \end{split}
\end{equation}
where $\Phi$ is a nonlinear vector function, $u(t)$ is the vector of circuit variables (e.g. node voltages and 
charges or branch currents),  $\dot{u}(t)$ is the derivative of this vector with respect to time $t$. 
The circuit analysis consists in determining the solution of Eq.~\eqref{eq:algebrodiff} for the following 
different simulation domains: nonlinear quiescent point calculation (DC), time-domain large-signal 
solution of nonlinear differential algebraic equations (TRAN), linear small-signal 
frequency domain (AC), noise analysis (NOISE). SPICE circuit 
simulation programs are presented in \cite{nodal}, \cite{vlach_singhal} and \cite{mccalla}.

The Monte Carlo simulation in SPICE programs is performed on multiple 
model evaluations with randomly selected model process variables. These variables are noted 
$X$ in this paper. 
It is worth mentioning that the input parameters in models are in general the various 
design parameters (geometry of transistors), the process parameters (random fluctuations 
due to manufacturing process), and the environmental conditions (voltages conditions 
and temperatures). The  electrical  models of  all  major  semiconductor 
manufacturers  contain  statistical  models 
of process parameters  with the estimation  of central tendency  characteristics, such  
as  mean  value,  standard  deviation, even higher moments such as skewness  or  kurtosis.
In  the  context of circuit  simulation,  the  problems  can live in very large 
dimension. The number of parameters are in general proportional to the number of transistors.
In section~\ref{sec:examples} some circuit examples will be presented where the dimension
of $X$ ranges from a few tens to several thousands components.

A circuit performance may be any electrical characteristic that is extracted from the 
simulation. It will be noted $H = h(X)$. In our context, the simulator acts as a black-box. 
We do not know anything a priori about the intrinsic nature of the   relation  between   
the   measured   output   and   the   statistical parameters. Typical measures of 
performance may be the time delay for signal propagation across a circuit, the gain 
of an amplifier or the oscillation frequency of an oscillator. 


\subsection{Yield losses and probability of failure}
\label{sec:yield}
The yield of  an  Integrated  Circuit  (IC) plays a major role in the manufacturing cost of an 
Application-Specific IC (ASIC). This concept typically enters in a manufacturing flow called 
“Yield Learning Flow”. In such a flow, the yield is defined to be the proportion of the manufactured 
circuits that satisfy both the design engineers and the process engineers specifications. 
The yield of a manufactured product should increase from the initial design to the 
large-volume production. Yield is therefore related to the cost of production. In 
nanoscale technologies, yield losses can be entirely caused by the uncontrolled and   
unavoidable factors  like random dopant fluctuation and line edge roughness 
(see \cite{agar:nass:08}). And each new technology node  tend  to  make  the  
problem   more difficult, since the size of the transistors within 
ICs has shrunk exponentially, following Moore's Law. Over the past decades, 
the typical MOS transistors channel lengths were once several micrometers (Intel 8086 CPU 
launched in 1978 was using a 3.2 $\mu$m process), but 
today ICs  are incorporating MOS with channel lengths of tens of nanometers. 
This empirical law predicted that the number of transistors would double 
every 2 years due to shrinking transistor dimensions and additional technical improvements. 
As a consequence, some authors predicted that power consumption per unit area would 
remain constant, and that that computer chip performance would roughly double every 18 months. 
The leading manufacturers continue to satisfy this empirical trend. In 2020, TSMC  
have already started mass producing 5 nanometer (nm) chips for companies including Apple, 
Huawei and Qualcomm. TSMC and Intel both have planned 2 nm products on their road-maps, 
with TSMC scheduling earliest production for 2023 and Intel for 2029. In May 2021, 
IBM announced the development of the first chip with 2nm technology.


Yield is interpreted as the probability of failure, rather than the probability of success. 
In some applications, it is much simpler to determine the yield of a single type of circuit 
which is  repeated  a  very  large  number  of  times, say  several millions. A
typical and important example is Static Random-Access Memory (SRAM) manufacturing yield. The simple 
block in SRAM circuits is the bitcell. Its role is to store one elementary bit of data using latching 
circuitry (flip-flop). The probability of 
failure is defined here to be the probability of a wrong operation (read, write, or retention) 
occurring in an SRAM circuit. For such problems, the yield may be quantitatively appreciated 
by considering an explicit value for the acceptable quality limit (AQL). For example, an AQL 
for 1-Megabit SRAM bit-cell circuits could be 100 ppm. If production of one million such 
circuits is targeted, the manufacturing process cannot allow more than 100 defective bit-cells 
globally. This means that the probability of failure must be at most:
$$
\frac{100 \text{\,-defective bit-cells}}{10^6 \text{\,-circuits} \times 10^6 \text{\,-bits}} = 10^{-10}
$$
This relations shows that yield drops exponentially with the number of elementary blocks, and  
imposes  extremely  small  values for where a large number of repetitions is used.

The practical problem of estimating  failure probabilities for a large class   of   circuits
is then  translated   into  a   general   mathematical framework. The goal is to compute the expectation $\E[\varphi(X)]$
where the function $\varphi$ writes as an indicator function $\varphi(x) = \ind{h(x) \ge \gamma}$ 
with $x \in \R^d$, where $h$ is  a circuit  performance  of  interest (evaluated through SPICE simulations) 
and $\gamma$ is  an upper  bound  specified by the circuit designer. Note also the practical importance of
the inverse problem, e.g. the estimation of the quantile $\gamma = \Gamma(p)$ given the probability of failure $0 < p < 1$ such that
the quantile function $\Gamma$ is defined as follows:  
\begin{equation}
  \label{eq:quantilefunc}
  \Gamma(p) = \inf \lbrace y \in \R : p \le F_H(y)\rbrace
\end{equation}
in terms of the output measure distribution $F_H$. The quantile problem is fundamental in verification flows where the circuit 
designers have to validate the design of complete library of standard cells. The circuits must satisfy 
prescribed specifications, for example the time delay must be less than some upper bound $\gamma \le \gamma_u$.
We will show, with the examples, how this quantile estimator can be obtained in a fully automatic way with our algorithm.

\vspace{1em}

The rest of the paper unfolds as follows. In Section~\ref{sec:is}, we present importance sampling in a Gaussian framework and further enhance it for very rare events. Then, in Section~\ref{sec:saa-default}, we apply the methodology developed before to the computation of a default probability and its expected shortfall. In Section~\ref{sec:stratif}, we explain how importance sampling and stratified sampling can be coupled to improve the performance of our algorithm on some particularly demanding cases. Finally, Section~\ref{sec:app} presents intensive numerical experiments on real electronic circuits.

\section{Importance sampling}
\label{sec:is}
\subsection{General framework}

Assume the random vector $X$ has a density function $g : \R^d \longmapsto \R_+$.  Consider a family of density functions $(f(\cdot; \t))_{\t \in \R^p}$ parametrized by a finite dimensional parameter $\t \in \R^p$ satisfying $\supp(g) \subset \supp(f(\cdot; \t))$ for all $\t \in \R^p$. Then, we can write
\begin{equation}
  \label{eq:is}
  \E[\varphi(X)] = \int_{\supp(g)} \varphi(x) h(x) dx = \int_{\supp(f(\cdot; \t))} \varphi(x) \frac{h(x)}{f(x;\t)} f(x;\t) dx = \E\left[ \varphi(Y) \frac{h(Y)}{f(Y;\t)} \right]
\end{equation}
where $Y$ is a random vector with density function $f(\cdot; \t)$. 

From a practical point of view, one has to find the best density function $h$ in order to maximize the accuracy of the Monte Carlo estimator. First, we need to make precise in which sens ``best'' has to be understood. 

As the convergence of the Monte Carlo estimator is governed by the central limit theorem, it is quite natural to try to find the density function $f(\cdot, \t)$ which minimizes best the variance of the estimator
\[
\Var\left( \varphi(Y) \frac{h(Y)}{f(Y; \t)} \right) = \E\left[ \varphi(Y)^2 \left(\frac{h(Y)}{f(Y;\t)}\right)^2  \right] - \left( \E[\varphi(X)] \right)^2
\]
Only the second moment depends on the density function $g$ and moreover using Eq.~\eqref{eq:is} again, we obtain
\[
 \E\left[ \varphi(Y)^2 \left(\frac{h(Y)}{f(Y;\t)}\right)^2  \right] =
 \E\left[ \varphi(X)^2 \frac{h(X)}{f(X;\t)}  \right].
\]
If the optimization problem $\min_\t \E\left[ \varphi(X)^2 \frac{h(X)}{f(X;\t)}  \right]$ is well-posed (strongly convex for instance), we can solve this problem directly.

\subsection{The Gaussian case}
\label{sec:meanshift}

Assume the density function $g$ stands for the standard Gaussian density function in $\R^d$.  Consider a family of Gaussian distributions with values in $\R^d$ parametrized by their expectations
\[
f(x;\t) = \frac{1}{(2 \pi)^{d/2}} \expp{-\norm{x - \t}^2/2} \quad \mbox{for }\t \in \R^d.
\]
Then, the importance sampling formula Eq.~\eqref{eq:is} writes as
\begin{equation}
  \label{eq:isgaussien}
  \E[\varphi(X)] = \E\left[ \varphi(Y) \frac{h(Y)}{f(Y;\theta)} \right] =
  \E\left[ \varphi(X+\theta) \expp{- X \cdot \t -\norm{\t}^2/2} \right].
\end{equation}
The second moment can be also written in a simplified way
\begin{equation}
  \label{eq:vgaussien}
  v(\t) = \E\left[ \left(\varphi(Y) \frac{h(Y)}{f(Y;\t)}\right)^2 \right] = \E\left[ \varphi(X)^2 \expp{- X \cdot \t +\norm{\t}^2/2} \right]
\end{equation}
where the last equality comes from the application of Eq.~\eqref{eq:isgaussien} with $-\t$. 
\begin{proposition}
  \label{prop:conv-gauss} Assume that for all $\theta \in \R^d$,
  $\E[\varphi(X)^2 \expp{-\theta \cdot X}] < +\infty$ and that $\P(\varphi(X) \neq 0) >0$. Then, the function $v$ is infinitely differentiable and strongly convex. Moreover, we can exchange derivation and expectation.
\end{proposition}
From this proposition, we can write
\begin{align}
  \label{eq:nablav}
  \nabla v(\t) & =  \E\left[ (\t - X) \varphi(X)^2 \expp{- X \cdot \t +\norm{\t}^2/2} \right] \\
  \label{eq:nabla2v}
  \nabla^2 v(\t) & =  \E\left[ (I_d + (\t - X)' (\t - X)) \varphi(X)^2 \expp{- X \cdot \t +\norm{\t}^2/2} \right].
\end{align}
Hence, the value $\t^*$ minimizing the variance is uniquely defined by $\nabla v(\t^*) = 0$. 
As the Hessian matrix is tractable at roughly the same cost as the gradient itself, one can implement a Newton algorithm with optimal descent direction to minimize $v$.

\subsection{Existing approaches in EDA}
\label{sec:others}

Importance Sampling is used with great success for many decades in multiple 
scientific fields, including structural safety engineering, aeronautics, communication networks, 
finance, nuclear safety. The world of EDA is actually rather behind, comparatively. 
IS with Gaussian mixtures has been used noticeably in the context of SRAM simulation 
(see~\cite{mixture_sram}).

The mathematics of the 
basic idea are rather straightforward, although some users have difficulties with this 
notion of using a modified distribution. The performance of any IS strategy heavily 
depends on the chosen instrumental distribution, and implementation must be very thorough. 
However, incorrect choice of the instrumental distribution
(as demonstrated in~\cite{hocevar}) or bad implementation can lead to poor performance, 
and even to an increase of the variance. 

However,   at   least   compared   to   the   field   of mechanical  engineering,  
the  relation  between  the  process parameters  and the  output  is  
more  often  mildly  non-linear  than  not.  Given  the difficulty,  many  
approaches  that yield  only  an  approximation  of the  probability  or  quantile  have  been  tried.  
Machine  Learning techniques  can  help  as  well.  
None  of  these  approaches  can  avoid the  curse  of  dimensionality  (see  for  example \cite{solido}, \cite{blockade}),  
and  the validity  of  the  approximated  confidence  intervals,  
when  they  can be estimated, is questionable.

\subsection{Algorithm for the mean shift approach in the Gaussian case}
\label{sec:saa-gaussian}

We consider the framework detailed in Section~\ref{sec:meanshift} relying on the variance criterion to determine the optimal parameter $\t^*$ uniquely defined by $\nabla v(\t^*) = 0$ thanks to Proposition~\ref{prop:conv-gauss}.  The idea of the algorithm is to replace all the expectations by empirical means and then deterministic optimization techniques.

\begin{algo}[Adaptive Sample Average Monte Carlo]\hfill
  \begin{enumerate}
    \item Draw $X_1, \dots, X_{n}$ $n$ i.i.d. samples following the normal distribution
    \item Compute  the minimizer $\t_{n}$ of
    \begin{equation*}
      v_{n}(\t) = \inv{n} \sum_{i=1}^{n}  \varphi^2(X_i) \expp{-\t \cdot X_i + \frac{\abs{\t}^2}{2}}
    \end{equation*}
    by solving $\nabla v_n(\t) = 0$ using Newton's method.
    \item Draw $\bar X_1, \dots, \bar X_{m}$ $m$ i.i.d. samples following the normal distribution and independent of $(X_1, \dots, X_{m})$.
    \item Compute the expectation $\E(\varphi(X))$ by Monte Carlo using this new sample
    \begin{equation*}
      M_{m}(\t_{n}) = \inv{m} \sum_{i=1}^{m} \varphi(\bar X_i+\t_{n}) \expp{-\t_{n} \cdot \bar X_i - \frac{\abs{\t{n}}^2}{2}}.
    \end{equation*}
\end{enumerate}
\end{algo}
The number of samples $n$ used for solving the optimization problem may differ from the number 
of samples $m$ used in the Monte Carlo procedure. The way to balance the number of samples between
these two steps has to be further studied. For $n$ large enough,  $v_n$ inherits its strong convexity
and differentiability from the ones of $v$. 

The convergence of the Sample Average approximation $\t_n$ is known to converge to $\t^\s$ and to satisfy a central limit theorem under integrability conditions. We introduce the following condition
\begin{equation}
  \exists \delta>0,\;\E[\varphi^{4 + \delta}(X)]<+\infty.
  \label{integ2}
\end{equation}

\begin{proposition}
  \label{prop:convtheta} Assume that for all $\theta \in \R^d$, $\E[\varphi(X)^2 \expp{-\theta \cdot X}] < +\infty$ and that $\P(\varphi(X) \neq 0) >0$. Then, $\t_n$ and $v_n(\t_n)$ converge a.s. to $\t^\star$ and $v(\t^\star)$ as $n\to\infty$. If, moreover, \eqref{integ2} holds, then $\sqrt{n}(\t_n-\t^\star) \xrightarrow[]{\cl}  \cn_{d'}(0,\Gamma)$ where 
  \[\Gamma=[\nabla^2v(\t^\star)]^{-1}\cov \left((\t^\star-X)\varphi^2(X)e^{-\t^\star\cdot X+ \frac{|\t^\star|^2}{2}}\right) [\nabla^2v(\t^\star)]^{-1}.\]
\end{proposition}
We refer to \cite{jour:lelo:09} and \cite{Rubinstein} for a proof of this result. \\
The convergence of the sequence $M_m(\t_n)_{m,n}$ is governed by the following theorem.  We can state the following result.
\begin{proposition}
  \label{prop:is-tcl}
  Let $n(m)$ be an increasing function of $m$ tending to infinity with $m$.  We assume the Assumptions of Proposition~\ref{prop:convtheta}. The following results hold.
  \begin{enumerate}[(i.)]
    \item The estimator $M_m(\t_{n(m)})$ is an unbiased and convergent estimator of $\E[\varphi(X)]$;
    \item If the function $\varphi$ is almost everywhere continuous and for some $\eta>0$ and some compact neighborhood ${\mathcal V}$ of $\t^*$, $ \sup_{\t \in {\mathcal
      V}} \E\left[ |\varphi(X + \t)|^{2+\eta} \right] < +\infty$, then $\sqrt{m} (M_m(\t_{n(m)}) - \E[\varphi(X)]) \xrightarrow[]{\cl}  {\cal N}_{d}(0,(\sigma^\s)^2)$ where
      \[
      (\sigma^\s)^2 = \Var \left(\varphi(X + \t^\s) \expp{- \t^\s \cdot X - \frac{|\t^\s|^2}{2}} \right) = v(\t^*) - \E[\varphi(X)]^2.
      \]
  \end{enumerate}
\end{proposition}
\begin{proof}
  The proof of \emph{(i.)} can be found in~\cite{is-jump}. \\

  We only give the guidelines to prove \emph{(ii.)}.
  \begin{equation*}
    \sqrt{n} ( M_n (\theta_{n(m)}) - \E[\varphi(X)]) = \sqrt{n} ( M_{n}(\t_\s) - \E[\varphi(X)]) + \sqrt{n} ( M_{n} (\t_{n(m)}) -  M_{n}(\t_\s)) 
  \end{equation*}
  From the standard central limit theorem, $\sqrt{n} ( M_{n}(\t_\s) - \E[\varphi(X)]) \xrightarrow[n \rightarrow +\infty]{\cl} \cn(0, v(\t_\s) - \E[\varphi(X)]^2)$. Therefore, it is sufficient to prove that $\sqrt{n} ( M_{n}(\t_{n(m)}) -  M_{n}(\t_\s)) \xrightarrow[n \rightarrow +\infty]{\P} 0$. Let $\varepsilon>0$.
  \begin{multline}
    \label{eq:tcl-split}
    \P\left( \sqrt{n} \abs{ M_{n}(\t_{n(m)}) -  M_{n}(\t_\s)} > \varepsilon \right) \le  \P( n(m)^{1/4} \abs{(\t_{n(m)}) - (\t_\s)} > 1) \\ + \frac{n}{\varepsilon^2} \E\left[ \abs{ M_{n}(\t_{n(m)}) -  M_{n}(\t_\s)}^2 \ind{\abs{\t_{n(m)} - \t_\s} \le n(m)^{-1/4} } \right].
  \end{multline}
  We deduce from Proposition~\ref{prop:convtheta}, that $\P( n(m)^{1/4}\abs{\t_{n(m)} - \t_\s} > 1) \longrightarrow 0$.  We introduce
  \begin{align*}
    Q(\t) = \expp{- \t \cdot \bar X^1 - \frac{|\t|^2}{2}}.
  \end{align*}
  Conditionally on $\t_{n(m)}$, $ M_{n}(\t_{n(m)}) -  M_n(\t_\s)$ is a sum of i.i.d. centered random variables. Then, it is sufficient to monitor differences as 
  \begin{align*}
    \abs{ \varphi(\bar X^1 +\t_\s) Q(\t_\s) - \varphi(\bar X^1 +\t_{n(m)}) Q(\t_{n(m)})}^2 
  \end{align*}
  on the set $\left\{\abs{\t_{n(m)} - \t_\s} \le n(m)^{-1/4} \right\}$, which is a subset of $\cv$ for large enough $n$.  As $\phi$ is almost everywhere continuous, we deduce that the above difference a.s. goes to $0$ when $n$ tends to infinity. Let $\delta < \eta$, the conditional independence combined with Hölder's inequality yield that for large enough $n$
  \begin{align*}
    & \E\left[ \abs{\varphi(\bar X^1 +\t_{n(m)}) Q(\t_{n(m)})}^{2(1+\delta)} \ind{\t_{n(m)} \in \cv} \right] \\
    & \le \sup_{\t \in \cv} \E\left[ \abs{\varphi(\bar X^1 +\t) Q(\t)}^{2(1+\delta)} \right] \\
    & \le \sup_{(\t, \lambda) \in \cv} \E\left[ \abs{\varphi(\bar X^1 +\t)}^{2(1+\eta)}\right]^{\frac{1+\delta}{1+\eta}} 
    \E\left[ Q(\t)^{\frac{2(1+\delta)(1+\eta)}{\eta-\delta}} \right]^{\frac{\eta-\delta}{1+\eta}}.
  \end{align*}
  From the local boundedness of $\t \longmapsto \E[\varphi(X+\t)]$, we deduce that
  \begin{align*}
    \sup_n \E\left[ \abs{\varphi(\bar X^1 +\t_{n(m)}) Q(\t_{n(m)})}^{2(1+\delta)} \right]  < \infty,
  \end{align*}
  which proves the uniform integrability of the family $(|\varphi(\bar X^1 +\t_{n(m)}) Q(\t_{n(m)})|^{2})_n$. Then, we deduce that the second term in~\eqref{eq:tcl-split} tends to zero.
\end{proof}
An immediate consequence of this result is that the accuracy of a  Monte Carlo algorithm based on $M_m(\t_{n(m)})$ can be monitored using the standard theory of confidence intervals. \\

A key issue from a practical and computational point of view is to fix $n(m)$ for a given value of $m$. As $n(m)$ has not impact on the convergence rate but only on how the empirical variance is close to the optimal variance. Hence, we are actually more interested in the convergence of $v_n(\t_n)$ to $v(\t^\s)$ than in the one of $\t_n$ itself. Monitoring the convergence of $v_n(\t_n)$ would give us a hint on how to choose $n(m)$.

\begin{proposition}
  \label{prop:tclvnthetan}
  Under the assumptions of Proposition~\ref{prop:convtheta},
  $\sqrt{n}(v_n(\t_n)-v(\t^\star)) \xrightarrow[]{\cl}  {\cn}(0,(\gamma^\s)^2)$ where
 \[(\gamma^\s)^2=\cov \left(\varphi^2(X)e^{-\t^\star\cdot X+ \frac{|\t^\star|^2}{2}}\right).\]
\end{proposition}
\begin{proof}
  \begin{align*}
    \sqrt{n} (v_n(\t_n) - v(\t^\s)) & = \sqrt{n} (v_n(\t_n) - v_n(\t^\s)) + \sqrt{n} (v_n(\t^\s) - v(\t^\s))
  \end{align*}
  By the standard central limit theorem, $\sqrt{n} (v_n(\t^\s) - v(\t^\s))$ converges in distribution to a normal random variable with mean $0$ and variance $(\gamma^\s)^2$.

  To handle the term $\sqrt{n} (v_n(\t_n) - v_n(\t^\s))$, we use the following Taylor expansion
  \begin{align*}
    \sqrt{n} (v_n(\t_n) - v_n(\t^\s)) = \sqrt{n} \nabla v_n(\tilde \theta_n) (\t^\s - \t_n)
  \end{align*}
  where $\tilde \theta_n$ is a point on the segment joining $\t^\s$ and $\t_n$. Since $\nabla v_n$ converges locally uniformly to $\nabla v$ (see~\cite{jour:lelo:09}), $\nabla v_n(\tilde \theta_n)$ converges a.s. to $\nabla v(\t^\s) = 0$. Then, we deduce from the convergence in distribution of $\sqrt{n} (\t_n - \t^\s)$ (see Proposition~\ref{prop:convtheta}) that
  $\sqrt{n} (v_n(\t_n) - v_n(\t^\s))$ goes to zero in probability. 
\end{proof}
As usual, if we consider an estimator $\gamma_n^2$ of $(\gamma^\s)^2$, we obtain $\frac{\sqrt{n}}{\gamma_n} (v_n(\t_n)-v(\t^\star)) \xrightarrow[]{\cl}  {\cn}(0,1)$, which can be used in practice to determine the magnitude of $n$. A natural estimator of $\gamma_n$ can be computed online within the minimization algorithm
\begin{align*}
  \gamma_n^2 = \inv{n} \sum_{i=1}^{n} \varphi^4(X_i) \expp{-2 \t_n \cdot X_i
  + \abs{\t_n}^2} - v_n(\t_n)^2.
\end{align*}
The convergence of $\gamma_n^2$ to $(\gamma^\s)^2$ ensues from the locally uniform strong law of large numbers for the sequence of random functions $\t \mapsto \varphi^4(X_i) \expp{-2 \t \cdot X_i + \abs{\t}^2}$, see~\cite[Lemma A1]{Rubinstein} or \cite[Corollary 7.10, page 189]{ledouxtalagrand}.

\subsection{Using importance sampling for \emph{very} rare event simulation}

Assume the function $\varphi$ writes as an indicator function $\varphi(x) = \ind{h(x) \ge \gamma}$ with $x \in \R^d$ where $\gamma$ makes the event almost never occur in practice, $\P(h(X) \ge \gamma)$ is much smaller than $1/n$. Then, the empirical updating formulae are not well defined and we will never manage to compute the mixture parameters straightaway. In such a situation, \cite{MR2329305} suggest to use a ladder with finite number of levels to reach level $\gamma$. Consider some levels $\gamma_1, \dots, \gamma_K = \gamma$ where $K$ is a priori unknown but computed online. In this scheme, the samples used at step $k$ are drawn according to the importance sampling distribution obtained from step $k-1$, which ensures that the level $\gamma_k$ is never too rare. \\

In this section, we focus only the mean shift approach. The multilevel approach is based on a step by step update of the importance density thus ensuring the targeted event is never too rare. At step $k$, the samples are generated according to the importance sampling density computed at step $k-1$. Before stating the algorithm, we need to extend Eq.~\eqref{eq:isgaussien} to a wider setting.

Consider two $\R^d-$valued normal random vector $Y$ and $Z$ with densities $f(\cdot; \theta_Y)$ and $f(\cdot; \theta_Z)$ respectively for some $\t_Y, \t_Z \in \R^d$. Then, from Eq.~\eqref{eq:is}, we get
\begin{align}
  \label{eq:isgaussien_generic}
  \E[\varphi(Y)] &= \E\left[\varphi(Z) \frac{f(Z; \theta_Y)}{f(Z,
      \theta_Z)}\right] \\
  \label{eq:vargaussien_generic}
  \E\left[\left(\varphi(Z) \frac{f(Z; \theta_Y)}{f(Z, \theta_Z)}\right)^2\right]
  & = \E\left[\varphi(Y)^2 \frac{f(Y; \theta_Y)}{f(Y, \theta_Z)}\right].
\end{align}
These two formulae can be made more explicit by expanding the density functions
\begin{align*}
  \E[\varphi(Y)] &= \E\left[\varphi(Z) \expp{- (\t_Z - \t_Y) \cdot Z +
      \inv{2}(\abs{\t_Z}^2  - \abs{\t_Y}^2)}\right] \\
  \E\left[\left(\varphi(Z) \frac{f(Z; \theta_Y)}{f(Z, \theta_Z)}\right)^2\right]
  & = \E\left[\varphi(Y)^2 \expp{- (\t_Z - \t_Y) \cdot Y +  \inv{2}
      (\abs{\t_Z}^2 - \abs{\t_Y}^2)}\right].
\end{align*}

The Algorithm can be written as
\begin{algo}[Multilevel importance sampling]\hfill\newline
  \label{algo:splitting-is}
  Fix $\t^{(0)}_n = 0$ in $\R^d$.
  \begin{enumerate}
    \item Draw some samples $X_1^{(k+1)}, \dots, X_n^{(k+1)}$ according to the distribution of $X^{(k)} \sim f(\cdot; \t^{(k)}_n)$.

    \item Using the previous samples, compute the $\rho$ quantile of $h(X^{(k)})$ and set it to $\gamma_{k+1}$ provided it is less than $\gamma$; otherwise set $\gamma_{k+1} = \gamma$.

    \item Compute the variance criterion
      \begin{align*}
        v^{(k+1)}_n(\t) = \frac{1}{n} \sum_{i=1}^n \ind{h(X^{(k+1)}_j) \ge \gamma} \expp{-(\t - \t^{(k)}_n) \cdot X^{(k+1)}_j + (\abs{\t}^2 - \abs{\t^{(k)}_n}^2) /2 } 
      \end{align*}
      Then, $\t^{(k+1)}_n = \arg\min_\t v^{(k+1)}_n(\t)$. 
    \item If $\gamma_k < \gamma$, return to step~1 and proceed with $k \leftarrow k+1$.
    \item Draw some new samples $\bar X_1, \dots, \bar X_m$ following the standard normal distribution and independent of the past. Approximate $\E[\varphi(X)]$ by
      \begin{align*}
        S_{m, n} = \inv{m} \sum_{j=1}^m \ind{h(\bar X_j + \t^{(k+1)}) \ge \gamma} 
        \expp{- \inv{2}\abs{\t^{(k+1)}_n}^2 - \t^{(k+1)}_n \cdot X_j}
      \end{align*}
  \end{enumerate}
\end{algo}
Using intermediate levels has no theoretical impact on the convergence of 
the algorithm, it is a practical though essential trick to circumvent the 
difficulty of sampling a rare event. If the levels $(\gamma_k)_k$ are chosen 
so that the events $\{h(X^{(k+1)}_j) \ge \gamma_k\}$ are never too rare, 
then a fairly small number of samples $n$ can be used. We know from 
Proposition~\ref{prop:is-tcl} that the convergence rate of $S_{m,n}$ is 
monitored by $\sqrt{m}$ so the computational effort should be put on $m$.

\paragraph{Computation of $\t^{(k+1)}_n$.} The vector $\t^{(k+1)}_n$ is defined as the solution of a minimization problem, which is known to be strongly convex.  Hence, the solution can be computed using the Newton algorithm with optimal descent direction given for finding the root of the gradient.
\begin{align*}
  \nabla v^{(k+1)}_n(\t) & = \frac{1}{n} \sum_{i=1}^n \ind{h(X^{(k+1)}_j) \ge \gamma} (\t - X^{(k+1)}_j) \expp{-(\t - \t^{(k)}_n) \cdot X^{(k+1)}_j + (\abs{\t}^2 - \abs{\t^{(k)}_n}^2) /2 } \\
  \nabla^2 v^{(k+1)}_n(\t) & = \frac{1}{n} \sum_{i=1}^n \ind{h(X^{(k+1)}_j) \ge \gamma} \left( I + (\t - X^{(k+1)}_j)(\t - X^{(k+1)}_j)'\right) \expp{-(\t - \t^{(k)}_n) \cdot X^{(k+1)}_j + (\abs{\t}^2 - \abs{\t^{(k)}_n}^2) /2 }.
\end{align*}
When $n$ goes to infinity, the Hessian matrix $\nabla^2 v^{(k+1)}_n(\t)$ converges to 
\begin{align*}
  \E\left[\ind{h(X^{(k+1)}) \ge \gamma} \left( I + (\t - X^{(k+1)})(\t - X^{(k+1)})'\right) \expp{-(\t - \t^{(k)}_n) \cdot X^{(k+1)} + (\abs{\t}^2 - \abs{\t^{(k)}_n}^2) /2 } \right]
\end{align*}
whose smallest eigenvalue is larger than
\begin{align*}
  & \E\left[\ind{h(X^{(k+1)}) \ge \gamma} \expp{-(\t - \t^{(k)}_n) \cdot X^{(k+1)} + (\abs{\t}^2 - \abs{\t^{(k)}_n}^2) /2 } \right] \\
  & \quad =   \E\left[\ind{h(X^{(k+1)}) \ge \gamma} \expp{-(\t - \t^{(k)}_n) \cdot X^{(k+1)}} \right]  \E\left[ \expp{(\t - \t^{(k)}_n) \cdot X^{(k+1)}}  \right] \ge   \E\left[\ind{h(X^{(k+1)}) \ge \gamma} \right]
\end{align*}
since
\begin{align*}
  \E\left[ \expp{(\t - \t^{(k)}_n) \cdot X^{(k+1)}} \right] = \expp{\inv{2} \abs{\t - \t^{(k)}_n}^2 + (\t - \t^{(k)}_n) \cdot \t^{(k)}_n} = \expp{\inv{2} (\abs{\t}^2  - \abs{\t^{(k)}_n}^2)}.
\end{align*}
Depending on the choice of the levels $(\gamma_k)_k$, the lower bound $\E\left[\ind{h(X^{(k+1)}) \ge \gamma} \right]$ can become quite small, which may lead to a badly performing Newton algorithm as the smallest eigenvalue of the Hessian matrix measures the convexity of the function and the less convex the function is, the slower the Newton algorithm converges. To circumvent this problem, we suggest to consider an alternative characterization of $\t^{(k+1)}_n$. The equation $\nabla v^{(k+1)}_n(\t^{(k+1)}_n) = 0$ can be rewritten
\begin{align*}
  \t^{(k+1)}_n - \frac{ \frac{1}{n} \sum_{i=1}^n \ind{h(X^{(k+1)}_j) \ge \gamma} X^{(k+1)}_j \expp{-(\t^{(k+1)}_n - \t^{(k)}_n) \cdot X^{(k+1)}_j }}{ \frac{1}{n} \sum_{i=1}^n \ind{h(X^{(k+1)}_j) \ge \gamma} \expp{-(\t^{(k+1)}_n - \t^{(k)}_n) \cdot X^{(k+1)}_j }} = 0.
\end{align*}
If we introduce the function $u^{(k+1)}_n$ defined by
\begin{align*}
  u^{(k+1)}_n(\t) = \inv{2} \abs{\t}^2 + \log\left( \frac{1}{n} \sum_{i=1}^n \ind{h(X^{(k+1)}_j) \ge \gamma} X^{(k+1)}_j \expp{-(\t - \t^{(k)}_n) \cdot X^{(k+1)}_j } \right) 
\end{align*}
then $\t^{(k+1)}_n$ can be seen as the root of $\nabla u^{(k+1)}_n(\t)$ and we can 
prove that the Hessian matrix $\nabla^2 u^{(k+1)}_n(\t)$ is uniformly lower bounded by $I_d$. 
We refer the reader to \cite{jour:lelo:09} for more details on this subtle transformation of the 
original problem to ensure a minimum convexity.

Note that in practice, one needs to solve the linear systems $\nabla^2 u^{(k+1)}_n(\t) \,\delta = -\nabla u^{(k+1)}_n(\t)$ 
at each iteration of the Newton algorithm. The Hessian matrix being a dense matrix, it is more interesting to use an 
iterative algorithm such as the conjugate gradient algorithm. 

The number $d$ of input variables may also be high. This dimension is directly linked to number of electronic
devices on a given circuit and specifically to the number of transistors. However this dimension does not 
impact the theoretical convergence properties of the approach. As this dimension grows, it becomes difficult 
to make sufficient progress during the multilevel algorithm. When the budget of simulation is limited, we experienced non 
convergence (and therefore failure) of the Importance Sampling algorithm if the number of 
variables $d$ is closed to $10^5$ variables. A simple idea is to look for the 
importance sampling parameter in a subspace $\left\{ A \vartheta \,:\, \vartheta \in \R^{d^\prime} \right\}$ where 
$A \in \R^{d \times d^\prime}$ is matrix with rank $d^\prime \le d$. 
In our situation, the matrix  $A$ needs to be identified automatically, and is a submatrix of the identity 
matrix after removing its $d-d^\prime$ columns corresponding to the 
unimportant variables. Our approach falls into the category of sequential deterministic methods. It 
produces the same subset on a given problem every time, by starting with a single solution (a variable subset) 
and iteratively add or remove dimensions until some termination criterion is met. This approach is clearly 
suboptimal, there is no guarantee for finding the optimal subset solution (because we do not examine all 
possible subsets). We made some tests to validate the approach with large scale problems see section~\ref{sec:large-scale}.

\section{Computing a default probability and its expected shortfall}
\label{sec:saa-default}

In this section, we assume that the function $\varphi$ is given by $\varphi(x) = \ind{h(x) \ge \gamma}$ 
where $h : \R^d \to \R$ and $a \in \R$.  We consider the problem of both estimating $\P(h(X) \ge \gamma)$ and 
$\E[h(X) | h(X) \ge \gamma]$, where $X$ is a standard normal random vector in $\R^d$. The quantity 
$\E[h(X) | h(X) \ge \gamma]$ is called conditional value at risk ($\Cvar$) or expected shortfall and 
we denote it by $\Cvar(h(X))$. The $\Cvar$ can be seen as a mesure of the dispersion of the distribution tail. 
In particular, a $\Cvar$ close to $\gamma$ means that the fail circuits are very concentrated and justifies 
the practitioners' habit to pick a single fail circuit at random when trying to analyse the reasons of a failure.
We will give an practtical example in section~\ref{sec:6t-sram-bitcell}.

The default probability is computed by applying the importance sampling approach developed in Section~\ref{sec:is} to its standard Monte Carlo estimator.

Note that $\Cvar(h(X)) $ can be written
\begin{align*}
  \Cvar(h(X)) = \frac{\E[h(X) \ind{h(X) > \gamma}]}{\P(h(X)> \gamma)}.
\end{align*}
Based on Section~\ref{sec:meanshift}, it is natural to consider the importance sampling version of the $\Cvar$ for $\t \in \R^d$
\begin{align}
  \label{eq:Cvar-ratio}
  \Cvar(h(X)) = \frac{\E\left[h(X + \t) \ind{h(X + \t) > \gamma} \expp{-\t \cdot X  - \abs{\t}^2/2}\right]}{\E\left[\ind{h(X + \t)> \gamma} \expp{-\t \cdot X  - \abs{\t}^2/2}\right]}.
\end{align}
Note that we could afford different importance sampling parameters for the numerator and the denominator, each being the solution of an optimization problem. Since, we at computing both $\Cvar(h(X))$ and $\P(h(X) \ge \gamma)$, we will anyway have to compute the optimal $\t$ for the denominator. Considering the computational cost of solving such an optimization problem, we believe it is not worth it using two different importance sampling parameters in~\eqref{eq:Cvar-ratio}.

Then, a standard estimator writes as
\begin{align}
  \label{eq:Cvarn}
  \Cvar_n(h(X)) = \frac{\sum_{i=1}^n h(X_i + \t) \ind{h(X_i + \t) > \gamma} \expp{-\t \cdot X_i  - \abs{\t}^2/2}}{\sum_{i=1}^n \ind{h(X_i + \t) > \gamma} \expp{-\t \cdot X_i  - \abs{\t}^2/2}}.
\end{align}
Although the estimations of $\P(h(X) \ge \gamma)$ and $\Cvar(h(X))$ are presented one and after the other, in practice they are computed simultaneously using the same runs of the simulator. Thus, the two estimators are actually obtained for the cost of a one only.

It is clear from the strong law of large number that $\Cvar_n(X) \to \Cvar(X)$ a.s. when $n \to \infty$. Due the random nature of an estimator, it has to come with a confidence interval, which relies on a central limit theorem.

\begin{proposition}
  \label{prop:tcl-Cvarn}
  The following convergence holds
  \begin{align}
    \label{eq:tcl-Cvarn}
    \sqrt{n}(\Cvar_n(h(X)) - \Cvar(h(X))) \xrightarrow[n \to \infty]{\cl} \cn(0, \sigma^2)
  \end{align}
  with
  \begin{align}
    \label{eq:Cvarn-variance}
    \sigma^2 &= \frac{\Var\left( Y^\t \ind{Y^\t > \gamma} \ce^\t \right)}{\P(Y > \gamma)^2} -2 \frac{\E[Y \ind{Y> \gamma}]}{\P(Y> \gamma)^3 } \E\left[ Y^\t \ind{Y^\t > \gamma} \left(\ce^\t \right)^2\right]\\
    & \qquad + \frac{\E[Y \ind{Y> \gamma}]^2}{\P(Y> \gamma)^3 } \left(\frac{\E\left[\ind{Y^\t > \gamma}\left(\ce^\t \right)^2\right]}{\P(Y > \gamma)} + \P(Y > \gamma)\right).
  \end{align}
\end{proposition}
\begin{proof}
  To keep the notation of the proof as compact as possible, we introduce $Y_i^\t = h(X_i + \t)$ and $\ce_i^\t = \expp{-\t \cdot X_i - \abs{\t}^2/2}$. When $i=1$, we simply drop the index. Similarly, for $\t = 0$, we write $Y = h(X)$.
\begin{align*}
  & \Cvar_n(Y) - \Cvar(Y) 
   = \frac{\sum_{i=1}^n Y^\t_i \ind{Y^\t_i > \gamma} \ce_i^\t}{\sum_{i=1}^n
   \ind{Y^\t_i> \gamma} \ce_i^\t} - \frac{\E[Y^\t \ind{Y^\t > \gamma}\ce^\t]}{\E[\ind{Y^\t> \gamma}\ce^\t ]} \\
  & = \frac{\sum_{i=1}^n Y^\t_i \ind{Y^\t_i > \gamma}\ce_i^\t}{\sum_{i=1}^n \ind{Y^\t_i> \gamma}\ce_i^\t} - 
  \frac{\E[Y^\t \ind{Y^\t > \gamma}\ce^\t]}{\sum_{i=1}^n \ind{Y^\t_i> \gamma}\ce_i^\t}
  + \frac{\E[Y^\t \ind{Y^\t > \gamma}\ce^\t]}{\sum_{i=1}^n \ind{Y^\t_i> \gamma}\ce_i^\t}
  - \frac{\E[Y^\t \ind{Y^\t > \gamma}\ce^\t]}{\E[\ind{Y^\t> \gamma}\ce^\t ]} \\
  & = \inv{\inv{n}\sum_{i=1}^n \ind{Y^\t_i> \gamma}\ce_i^\t} 
  \left(\inv{n}\sum_{i=1}^n Y^\t_i \ind{Y^\t_i > \gamma}\ce_i^\t - \E[Y^\t \ind{Y^\t > \gamma}\ce_i^\t]\right)
   \\
  & \qquad - \frac{\E[Y^\t \ind{Y^\t > \gamma}\ce^\t]}{\E[\ind{Y^\t> \gamma}\ce^\t ] \inv{n} \sum_{i=1}^n \ind{Y^\t_i> \gamma}\ce_i^\t }
  \left( \inv{n}\sum_{i=1}^n \ind{Y^\t_i> \gamma} \ce_i^\t - \E[\ind{Y^\t> \gamma}\ce^\t ] \right)  \\
  & = \inv{\inv{n}\sum_{i=1}^n \ind{Y^\t_i> \gamma}\ce_i^\t} 
  \left(\inv{n}\sum_{i=1}^n Y^\t_i \ind{Y^\t_i > \gamma}\ce_i^\t - \E[Y\ind{Y> \gamma}]\right)
   \\
 & \qquad - \frac{\E[Y\ind{Y> \gamma}]}{\P(Y> \gamma) \inv{n} \sum_{i=1}^n \ind{Y^\t_i> \gamma}\ce_i^\t }
  \left( \inv{n}\sum_{i=1}^n \ind{Y^\t_i> \gamma} \ce_i^\t - \E[\ind{Y> \gamma}]\right) 
\end{align*}
From the standard multi-dimensional central limit theorem, it is known that
\begin{align}
  \label{eq:tcl-nd}
  \sqrt{n}\left( \inv{n} \sum_{i=1}^n Y^\t_i \ind{Y^\t_i > \gamma}\ce_i^\t  -  \E[Y\ind{Y> \gamma}] , \inv{n}  \sum_{i=1}^n
  \ind{Y^\t_i> \gamma} \ce_i^\t - \P(Y> \gamma) \right) \xrightarrow[n \to \infty]{\cl} \cn(0, \Sigma)
\end{align}
where 
\begin{align*}
  \Sigma &= \cov((Y^\t \ind{Y^\t > \gamma}\ce^\t , \ind{Y^\t> \gamma}\ce^\t )) \\
  & = \begin{pmatrix}
    \Var(Y^\t \ind{Y^\t > \gamma} \ce^\t) & \cov(Y^\t \ind{Y^\t > \gamma}\ce^\t , \ind{Y^\t> \gamma}\ce^\t ) \\
    \cov(Y^\t \ind{Y^\t > \gamma}\ce^\t , \ind{Y^\t> \gamma}\ce^\t ) & \Var(\ind{Y^\t > \gamma} \ce^\t) 
  \end{pmatrix}
\end{align*}
This convergence in distribution combined with Slutsky's theorem yields that $\sqrt{n} (\Cvar_n(Y) - \Cvar(Y))$ converges in distribution to a normal random variable with mean $0$ and variance
  \begin{align*}
    \sigma^2 & = 
    {\begin{pmatrix}
        \inv{\P(Y > \gamma)} \\
        - \frac{\E[Y \ind{Y> \gamma}]}{\P(Y> \gamma)^2 }
    \end{pmatrix}}^\prime
    \Sigma
    \begin{pmatrix}
        \inv{\P(Y > \gamma)} \\
        - \frac{\E[Y \ind{Y> \gamma}]}{\P(Y> \gamma)^2 }
    \end{pmatrix} \\
    & = \frac{\E\left[ \left( Y^\t \ind{Y^\t > \gamma} \ce^\t \right)^2 \right]}{\P(Y > \gamma)^2} \\
    & \qquad
    + \frac{\E[Y \ind{Y> \gamma}]}{\P(Y> \gamma)^3 }\left( -2 \E\left[ Y^\t \ind{Y^\t > \gamma} \left(\ce^\t
      \right)^2\right] + \E\left[\ind{Y^\t > \gamma}\left(\ce^\t
    \right)^2\right] \frac{\E[Y \ind{Y> \gamma}]}{\P(Y > \gamma)}\right)  \\
    & = \frac{\Var\left( Y^\t \ind{Y^\t > \gamma} \ce^\t \right)}{\P(Y > \gamma)^2}
    -2 \frac{\E[Y \ind{Y> \gamma}]}{\P(Y> \gamma)^3 } \E\left[ Y^\t \ind{Y^\t > \gamma} \left(\ce^\t
    \right)^2\right]\\
    & \qquad
    + \frac{\E[Y \ind{Y> \gamma}]^2}{\P(Y> \gamma)^3 } \left(\frac{\E\left[\ind{Y^\t > \gamma}\left(\ce^\t
    \right)^2\right]}{\P(Y > \gamma)} + \P(Y > \gamma)\right). \qedhere
  \end{align*}
\end{proof}

\paragraph{Some comments on Proposition~\ref{prop:tcl-Cvarn}.}

It is interesting to analyse the convergence result stated in Proposition~\ref{prop:tcl-Cvarn} in light of the behaviour of an other estimator of $\Cvar(X)$. Let us introduce
\begin{align*}
  \overline{\Cvar}_n (h(X)) = \frac{\sum_{i=1}^n Y^\t_i \ce^\t_i \ind{Y^\t_i > \gamma}}{\P(X> \gamma)}.
\end{align*}
Clearly, $\E[\overline{\Cvar}_n(X)] = \E[Y | Y > \gamma]$, which assesses that $\overline{\Cvar}_n (h(X))$ is non biased. From, the standard central limit theorem, we deduce that
\begin{align*}
  \sqrt{n}(\Cvar_n(h(X)) - \Cvar(h(X))) \xrightarrow[n \to \infty]{\cl} 
  \cn( 0,  \bar \sigma^2) \mbox{ with } \bar \sigma^2 = \frac{\Var(Y^\t \ce^\t_i \ind{Y^\t > \gamma})}{\P(X
  > \gamma)^2}.
\end{align*}
Clearly,
\begin{align*}
  \bar \sigma^2 - \sigma^2 &= 2 \frac{\E[Y \ind{Y> \gamma}]}{\P(Y> \gamma)^3 } 
  \E\left[ Y^\t \ind{Y^\t > \gamma} \left(\ce^\t \right)^2\right] \\
  & \quad
  - \frac{\E[Y \ind{Y> \gamma}]^2}{\P(Y> \gamma)^3 } \left(\frac{\E\left[\ind{Y^\t > \gamma}\left(\ce^\t
  \right)^2\right]}{\P(Y > \gamma)} + \P(Y > \gamma)\right).
\end{align*}
For $\t= 0$, the difference becomes
\begin{align*}
  \bar \sigma^2 - \sigma^2 &= 2 \frac{\E[Y \ind{Y> \gamma}]}{\P(Y> \gamma)^3 } \E\left[ Y \ind{Y > \gamma} \right]
  - \frac{\E[Y \ind{Y> \gamma}]^2}{\P(Y> \gamma)^3 } \left(\frac{\E\left[\ind{Y > \gamma}
  \right]}{\P(Y > \gamma)} + \P(Y > \gamma)\right) \\
  & =  \frac{\E[Y \ind{Y> \gamma}]^2}{\P(Y> \gamma)^3 }  (1 - \P(Y > \gamma)) > 0.
\end{align*}

The limiting variance $\bar \sigma^2$ is larger than the one appearing in~\eqref{eq:tcl-Cvarn} (at least for $\t = 0$) but unlike $\overline{\Cvar}_n(h(X))$, the estimator $\Cvar_n(h(X))$ is biased, which is precisely the effect induced by replacing $\P(Y > \gamma)$ by $\inv{n} \sum_{i=1}^n \ind{Yi^\t_i > \gamma} \ce^\t_i$ in $\overline{\Cvar}_n(h(X))$ to obtain $\Cvar_n(h(X))$.

The gap between the two limiting variances may become impressively large. Consider the following toy example with $h(x) = x$, $X$ following a real valued standard normal distribution and $a = 1.5$. Take $\t = 0$. In this case, $\P(X > \gamma) \approx 0.067$ and we find $\sigma^2 \approx 2$ whereas $\bar \sigma^2 \approx 54$. 

In the case $\t = 0$, we can compute the bias.
\begin{proposition}
  The estimator $\Cvar_n(h(X))$ has bias $-\E[h(X) | h(X) > \gamma] \P(h(X) \le \gamma)^n$.
\end{proposition}
\begin{proof}
  We use the notation $Y = h(X)$. We define $T_n = \sum_{i=1}^n \ind{Y_i > \gamma}$ and the tuple $(Z_1, \dots, Z_{T_n})$ holds the values of $(Y_1, \dots, Y_n)$ for which the $Y_i > \gamma$. With these new random variables,
  \begin{align*}
    \Cvar_n(Y) = \inv{T_n} \sum_{i=1}^{T_n} Z_i.
  \end{align*}
  \begin{align}
    \label{eq:split}
    \E\left[  \inv{T_n} \sum_{i=1}^{T_n} Z_i  \right] = \sum_{t=1}^n  \E\left[  \inv{t} \sum_{i=1}^{t} Z_i \ind{T_n =t}  \right] = \sum_{t=1}^n  \E\left[  Z_1 \ind{T_n =t}  \right]. 
  \end{align}
  as the random variable $Z_i \ind{T_n = t}$ are identically distributed. Let us compute the joint distribution of $(Z_1, T_n)$. Let $z \in \R$ and $t \in \{1, \dots, n\}$
  \begin{align*}
    &\P(Z_1 > z, T_n = t) \\
    &= \P(\exists i_1,\dots, i_t \mbox{ s.t. } Y_{i_1} > z \vee \gamma,
    Y_{i_2} > \gamma,\dots, Y_{i_t} > \gamma, \forall i \notin \{i_1,\dots, i_t\} Y_i \le \gamma) \\
    &  = C_n^t \P(Y_1 > z \vee \gamma, Y_2 > \gamma,\dots, Y_t > \gamma, Y_{t+1} \le \gamma, \dots, Y_n \le \gamma) \\
    &  = C_n^t \P(Y > z \vee \gamma) \P(Y> \gamma)^{t-1} (1 - \P(Y> \gamma))^{n-t}.
  \end{align*}
  If $z \le \gamma$, $\P(Z_1 > z, T_n = t)  = \P(T_n=t) = C_n^t \P(Y> \gamma)^{t} (1 -
  \P(Y> \gamma))^{n-t}$.
  Note that $\P(Z_1 \ind{T_n = t} > z) = \P(Z_1 > z, T_n = t)$. Hence,
  \begin{align*}
    d\P(Z_1 \ind{T_n = t} > z) = 
    \begin{cases}
      0 & \mbox{if $z \le a$} \\
      C_n^t \P(Y> \gamma)^{t-1} (1 - \P(Y> \gamma))^{n-t} d \P(Y > z) & \mbox{if $z > \gamma$}
    \end{cases}
  \end{align*}
  Then, from \eqref{eq:split}, we deduce
  \begin{align*}
    \E[\Cvar_n(Y)] & =  \sum_{t=1}^n  \int_a^\infty z d\P(Z_1 \ind{T_n = t} \le z)  \\
    & = \sum_{t=1}^n  C_n^t \P(Y> \gamma)^{t-1} (1 - \P(Y> \gamma))^{n-t} \int_a^\infty z d \P(Y < z)\\
    & = \sum_{t=1}^n  C_n^t \P(Y> \gamma)^{t-1} (1 - \P(Y> \gamma))^{n-t} \E[Y \ind{Y> \gamma}] \\
    & = \E[Y | Y> \gamma] (1 - (1 - \P(Y> \gamma))^{n} ). \qedhere
  \end{align*}
  
\end{proof}

\section{Stratified sampling}
\label{sec:stratif}

An other well-known tool for rare event simulation is stratified sampling. This technique can lead to very impressive results provided one has a good a priori knowledge of the strata and of their respective weights in the sampling.

\subsection{Introduction to stratification}

Consider a partition $D_1, \dots, D_I$ of $\R^d$ and assume we know how to sample according to $\cl(X | X \in D_i)$ for $i=1,\dots, I$. Let $p_i = \P(X \in D_i)$. We assume that the strata are chosen such that $p_i >0$ for all $i$.  For a fixed $1 \le i \le I$, let $(X_i^{j}, 1 \le j \le N_i)$ be i.i.d.  samples according to $\cl(X | X \in D_i)$ where $N_i$ is the number of samples in stratum $i$. We assume that the samples in two different strata are independent.

We know that
\begin{align*}
  \E[\varphi(X)] = \sum_{i=1}^I \E[\varphi(X) | X \in D_i] \, p_i
\end{align*}
Hence, we can use the following Monte Carlo estimator to approximation
$\E[\varphi(X)]$
\begin{align*}
  \sum_{i=1}^I p_i \inv{N_i} \sum_{j=1}^{N_i} \varphi(X^{j}_{i}).
\end{align*}
Assume the total number of samples is $N$, then the variance of this estimator is minimum for $N_i = q_i / N$ with 
\begin{align}
  \label{eq:alloc}
  q_i = \frac{p_i v_i}{\sum_{\ell=1}^I p_\ell v_\ell}
\end{align}
where $v_i^2 = \Var(\varphi(X^{(i)}))$. Using this optimal allocation yields to an estimator with minimum variance given by $\frac{v_*^2}{N}$ with
\begin{equation*}
  v_*^2 = \left( \sum_{i=1}^I p_i v_i \right)^2.
\end{equation*}
However, this allocation formula can be hardly used in practice as the $v_i's$ are usually unknown. To overcome this difficulty, \cite{EtoJou10} proposed an adaptive scheme to learn the $(v_i)_{i=1,\dots,I}$ and then the $(q_i)_{i=1,\dots,I}$.

\subsection{Importance sampling as a stratifying direction}
\label{sec:strat-is}

The key difficulty in applying stratified sampling techniques is to know which strata to be used. 
\cite{glas_strati} suggested to use the optimal importance sampling vector as a direction along 
which to stratify. Assume some optimal shift $\mu^\s$ has already been computed. The vector 
$u = \frac{\mu^\s}{\norm{\mu^\s}}$ defines an hyperplane on which the strata are based. 
Consider an increasing sequence of levels $(a_i)_{0 \le i \le I}$ with $a_0 = -\infty$ 
and $a_I = +\infty$, we consider the partition $\cup_{i=1}^I D_i$ with 
$D_i = \{x \in \R^d \; : \; a_{i-1} \le u \cdot x < a_i\}$. Since $u \cdot X$ follows the 
standard normal distribution with values in $\R$, the levels $a_i$'s are usually chosen as 
quantiles of the standard normal distribution.  Then, simulating the normal random vectors 
conditionally to being in a given stratum can be easily implemented by using the following result.
\begin{lemma}
  \label{lem:conditional_sampling} Let $a < b$ be real numbers and $X \sim \cn(0,I_d)$. Let $U$ be a
   random variable with uniform distribution on $[0,1]$ and $Y \sim \cn(0,I_d)$ both independent of $X$. 
   Set $Z = \Phi^{-1}(\Phi(a) + U(\Phi(b) - \Phi(a)))$, where $\Phi$ is the cumulative distribution function
   of the standard normal distribution. Then,
  \begin{equation*}
    u Z + Y - u (u \cdot Y) \sim \cl(X | a \le u \cdot X \le b).
  \end{equation*}
\end{lemma}
A natural approach is to proceed in two steps
\begin{enumerate}
  \item \emph{Importance Sampling step}. Apply Algorithm~\ref{algo:splitting-is} to compute an almost 
  optimal change of measure with parameters $(\mu^\s, I_d)$.
  \item \emph{Stratified sampling}. Set $u = \frac{\mu^\s}{\norm{\mu^\s}}$.  Fix some levels 
  $(a_0, \ldots, a_I)$ based on the quantiles of the standard normal distribution and use the strata 
  defined by $D_i = \{x \in \R^d \; : \; a_{i-1} \le u \cdot x < a_i\}$ for $1 \le i \le I$ in the algorithm 
  proposed by~\cite{EtoJou10}.
\end{enumerate}

\section{Rare events simulation for electronic circuits}
\label{sec:app}

\subsection{Large scale analytic problems}
\label{sec:large-scale}

We validate in this section the Importance Sampling multilevel approach with large scale problems.
The test-cases were generated automatically for the purpose of this test, and not related to circuit simulation.

The response function is defined as $h(x) = a \cdot x_{A} + b \cdot x_{B}$, and $x = (x_A, x_B)$ is a vector of normal
variables $\cn(0,1)$. The coefficients $a$ and $b$ were defined such as $ \max_{i\in A} |a_i| \gg \max_{i \in B} |b_i|$. 
This makes the subset $A$ the set of important variables. The remaining $B$-components may be considered as noisy variables.

For our comparisons, we fixed the cardinality $|A| = 10$ and the vector $a$, and we gradually increased the dimension of 
the noisy components. The size of the batches is fixed at $1000$ runs.

\begin{center}
  \begin{tabular}{|c|c|c|c|}
    \hline
    $|B|$ & Prob. & CI@95\% & Nb. Runs \\
    \hline\hline
    1000            & $2.8039 \times 10^{-5}$ & $8.60$\%  &  $5000$    \\ 
    2000            & $2.9891 \times 10^{-5}$ & $8.91$\%  &  $5000$    \\ 
    10000           & $2.8239 \times 10^{-5}$ & $9.65$\%  &  $10000$    \\ 
    50000           & $2.9815 \times 10^{-5}$ & $9.83$\%  &  $20200$   \\ 
    \hline
  \end{tabular}
  \end{center}

The previous table indicates that the approach is rather not dependent on the input dimension of the problem. The 
differences, as the dimension grows, are due mainly to the difficulty of finding the 'optimal' subset of important
variables during the search of the final mean-shift. Increasing the batch size may help as well, because larger 
batches will improve the accuracy of the reduction dimension algorithms.

\subsection{Circuit examples}

We use systematically the same settings for the Importance sampling algorithm. The relative precision on the 
probability estimator is fixed at 10\%. And the size of each batch is 1000 runs, the total number of runs is 
therefore a multiple of this size. 
We report the ``Speedup'' in our tests as the fraction of the theoretical number runs of the plain MC over the effective
number of runs to obtain the same relative accuracy (based on on Central Limit theorem).
For our experiments, electrical simulations will be achieved with the commercial SPICE simulators 
Eldo\textsuperscript{\tiny\textregistered} and EldoRF from Siemens EDA.

\label{sec:examples}
\subsubsection{VCO LC Tank}
This is an example of an analog IP block showing the IS algorithm for low probability estimation and quantile estimation. The circuit 
is a simple LC-tank VCO oscillating at 1.8GHz (inspired from \cite{craninckx})

The circuit contains 28 Gaussian random variables. The oscillation frequency (FOSC) is obtained using the Steady-State Analysis of Eldo RF.
The phase noise (PNOISE) at 1 MHz offset from the carrier 
is also extracted using the Steady-State noise analysis of Eldo RF. Both the oscillation frequency and the phase noise are strongly non-Gaussian.
The distributions are heavily skewed (left and right respectively). This means that all kinds of Gaussian extrapolations using the estimated 
standard deviation with a few 100's runs are completely wrong. The results for the analysis of ISMC are mostly read in the following table.

\begin{center}
\begin{tabular}{|c|c|c|c|c|c|}
  \hline
  Measure & Tail & Prob & CI@95\% & Nb. Runs & Speedup \\
  \hline\hline
  FOSC             & Left  & $7.9849 \times 10^{-5}$ & $8.89$\%  &  $5000$   &   $1.2 \times 10^3$ \\
  PNOISE           & Right & $3.5240 \times 10^{-6}$ & $9.05$\%  &  $9000$   &   $1.5 \times 10^4$ \\ 
  \hline
\end{tabular}
\end{center}

The \texttt{Prob} column indicates the tail probability and \texttt{CI@95\%} the relative Confidence Interval at 95\% level. 
The \texttt{Nb. Runs} column gives the total number of runs for each estimator. The \texttt{Speedup} columns finally provides 
the speedup factor over a plain Monte Carlo simulation (for the same accuracy).  

Similarly, one gives the results for the quantile estimation where we fixed the tail probability at level $1.0 \times 10^{-4}$. 
 
\begin{center}
  \begin{tabular}{|c|c|c|c|c|c|}
    \hline
    Measure & Tail & Quantile & CI@95\% & Nb. Runs & Speedup \\
        \hline\hline
    FOSC             & Left  & $1.75195 \times 10^9$ & $0.06$\%  &  $4000$   &   $3.1 \times 10^2$ \\
    PNOISE           & Right & $-1.1239 \times 10^{2}$ & $0.16$\%  &  $4000$   &   $2.2 \times 10^2$ \\ 
    \hline
  \end{tabular}
  \end{center}

\subsubsection{6T SRAM Bitcell}
\label{sec:6t-sram-bitcell}
The ciruit is a classical six transistors SRAM bit cell designed
in CMOS 45nm. A subset of the device model parameters has variability information with 
Gaussian random variables. In this example, the total number
of statistical parameters is 36 (6 per device). The cell is configured for 
measuring the write delay. The purpose of the
Monte Carlo simulation is to estimate the probability for the write delay
to be larger than a certain threshold. This threshold might be used to
define a ``fail'' behavior for yield estimation. Formally, the problem is
to estimate the probability of failure and the associated confidence interval.

This kind of low probability estimation is difficult (not to say impossible)
with plain Monte Carlo, because the number of runs required to obtain a decent
accuracy (in terms of confidence interval on the probability) is huge.
In this particular example, the IS flow estimates a probability of $9.1893 \times 10^{-6}$
to within $9.99$\% (with 95\% confidence level) in 8000 runs, whereas it would
require approximately $4\times 10^{7}$ runs with plain Monte Carlo.

\begin{center}
  \begin{tabular}{|c|c|c|c|c|c|}
    \hline
    Measure & Tail & Prob & CI@95\% & Nb. Runs & Speedup \\
    \hline\hline
    WDELAY             & Right  & $9.1893 \times 10^{-6}$ & $9.99$\%  &  $8000$   &   $5.2 \times 10^3$ \\
    \hline
  \end{tabular}
  \end{center}

We give the detail of the multilevel Importance Sampling phase with this circuit. 
We recall that we need to use intermediate levels to reduce the 
difficulty of sampling a rare event. If the levels $(\gamma_k)_k$ are chosen 
so that the events $\{h(X^{(k+1)}_j) \ge \gamma_k\}$ are never too rare, 
then a fairly small number of samples $n$ can be used. The following table 
provides the first iterations of this exploration phase. 

\begin{center}
  \begin{tabular}{|c|c|c|c|c|}
    \hline
    Iteration (k) & Nb. Runs & $\gamma_k$ & $S_{m,n}(\gamma_k)$ & CI@95\%  \\
    \hline\hline
    0         & 1000     & $3.99328 \times 10^{1}$ & $1.51331 \times 10^{-3}$ &  $21.9$\% \\
    1         & 1000     & $4.69753 \times 10^{1}$ & $1.38538 \times 10^{-4}$ &  $27.0$\%  \\     
    2         & 1000     & $5.31605 \times 10^{1}$ & $8.51455 \times 10^{-6}$ &  $34.7$\% \\
    \hline
  \end{tabular}
  \end{center}

We now provide some results on CVaR computation for this circuit. As noted above, comparing the threshold $\gamma$
and CVaR may give a hint about the dispersion of 'Failed' population of runs.  

\begin{center}
  \begin{tabular}{|c|c|c|c|c|c|}
    \hline
    Measure & $\gamma$ & CVaR & CI@95\% & Nb. Runs & Speedup \\
    \hline\hline
    WDELAY             & $5.3120 \times 10^1$&  $5.3388 \times 10^1$  & $10.1$\%  &  $9000$   &   $5.1 \times 10^3$ \\
    \hline
  \end{tabular}
  \end{center}

  It is important to note here that the CVaR is a by-product of the IS flow which is mainly drived by the computation 
  of the failure probability. The computation of the CVaR is based on the same samples. Here practitioners
  may conclude that picking one point of the failed runs provides an accurate example for 'debugging' or 
  understanding the reasons that lead to this fail point. The analysis of the largest components of the
  mean-shift may help further to identify to variables and therefore the electronic devices at the origin of the
  typical failure point.  

  \subsubsection{StdCell LH}
This circuit is flip-flop or latch circuit extracted from a library of standard cells. It is a circuit 
that can be used to store state information, and a fundamental storage element in sequential logic
of digital electronics systems used in computers, and many types of systems. The number of random 
variables is fixed at 66. The ``LAB3'' output represents the measure of a delay (in seconds). 

\begin{center}
  \begin{tabular}{|c|c|c|c|c|c|}
    \hline
    Measure & Tail & Prob & CI@95\% & Nb. Runs & Speedup \\
    \hline\hline
    LAB3           & Right & $1.2493 \times 10^{-9}$ & $8.05$\%  &  $7000$   &   $6.8 \times 10^7$ \\ 
    \hline
  \end{tabular}
\end{center}

The speedup here is quite impressive, but note that in general the testbenches used for such 
standard cells contain multiple output measures. Because the samples cannot be mutualized accross 
the different estimators, the total number of runs will be a multiple of the computed quantiles 
and/or probabilities. 

\subsubsection{StdCell MUX4}
This circuit a multiplexer (or mux) is another fundamental block extracted from a library of standard cells. 
The number of random variables is fixed at 116. The output response chosen for this experiment is a measure 
of delay (in seconds) on a critical path. It is therefore a time-domain analysis.

We are interested here in the computation of quantiles at some predefined failure probabilities.

\begin{center}
  \begin{tabular}{|c|c|c|c|c|c|}
    \hline
    Measure & Failure Prob. & Quantile & CI@95\% & Nb. Runs & Speedup \\
    \hline\hline
    LAB179  & $1.34990 \times 10^{-3}$ & $2.55684 \times 10^{-10}$ & $0.72$\%  &  $4000$   &   $3.2$ \\ 
    \hline
       -    & $9.86588 \times 10^{-10}$ & $2.78912 \times 10^{-10}$ & $0.22$\%  &  $6000$   &   $3.9 \times 10^6$ \\
    \hline
  \end{tabular}
\end{center}

As we noted in the previous example, if multiple circuit output need to be characterized then the 
speedup $3.2$ obtained for the first failure probability (approx. $1.3 \times 10^{-3}$) may 
not compensate the total number of runs required for  estimating each quantile individually. 
A plain Monte Carlo simulation may do the job because the quantiles
can be computed on the same samples. This option may not be possible when very low failure probabilities
are searched. In such cases, the IS algorithm may prevail over the plain MC.


\subsubsection{Memory Cell}
This circuit represents a block of memory simulated with its peripheral circuitry. The number of random variables 
is fixed at 2096. The simulation output is here the read delay measured in seconds.

\begin{center}
  \begin{tabular}{|c|c|c|c|c|c|}
    \hline
    Measure & Tail & Prob & CI@95\% & Nb. Runs & Speedup \\
    \hline\hline
    READ10           & Right & $3.4506 \times 10^{-8}$ & $8.96$\%  &  $9000$   &   $1.5 \times 10^6$ \\ 
    \hline
  \end{tabular}
\end{center}

This last example is provided to demonstrate the effectiveness of IS algorithm on a quite large 
size problem. This is a time-domain analysis which requires several minutes for a nominal simulation on a 
modern architecture. Our IS algorithm can run in full parallel mode, taking advantage of compute farms.

\bibliographystyle{abbrvnat}
\bibliography{rare-biblio}
\end{document}